\documentclass[12pt,a4paper,oneside]{article} 
\usepackage{thesis}
\usepackage{rotating,graphicx,amsbsy,amsfonts,epsfig,amsmath}
\usepackage{amssymb,psfrag,lscape,epsf,color,listings}
\usepackage{amsthm}
\usepackage{latexsym,graphicx,psfrag}
\usepackage{a4wide}
\usepackage{enumerate}
\usepackage{comment}
\usepackage{graphicx}
\usepackage{makeidx}
\usepackage{authblk}
\usepackage{hyperref}
\input xy  
\xyoption{all}

\newtheorem{theorem}{Theorem}[section]

\newtheorem{conjecture}[theorem]{Conjecture}
\newtheorem{lemma}[theorem]{Lemma}

\newtheorem{definition}[theorem]{Definition}

\newtheorem{example}[theorem]{Example}
\newtheorem{exercise}[theorem]{Exercise}

\numberwithin{equation}{section}

\def\S{\Sigma }

\def\<{\left<}
\def\>{\right>}

\newcommand{\be}{\begin{enumerate}}
\newcommand{\ee}{\end{enumerate}}
\newcommand{\bd}{\begin{description}}
\newcommand{\ed}{\end{description}}

\newcommand{\rad}{{\rm rad}}
\makeindex

\begin{document}
\title{Powerful numbers and the ABC-conjecture}
\author{David Cushing\\ J. E. Pascoe}

\maketitle
\begin{abstract}
The $abc$ conjecture is a very deep concept in number theory with wide application to many areas of number theory. In this article we introduce the conjecture and give examples of its applications. In particular we apply the $abc$ conjecture to the location of powerful numbers.
\end{abstract}

\section{Introduction}
The $abc$ conjecture has gotten a lot of press lately. This begs the questions:
\begin{enumerate}
	\item What does the $abc$ conjecture say?
	\item More importantly, what does the $abc$ conjecture do?
\end{enumerate}

Given a number $x$ with prime factorization 
	$$x= p_1^{a_1}p_2^{a_2}\ldots p_n^{a_n},$$
we define the {\it radical} of $x$ via the following formula:
	$${\rm rad}(x) = p_1p_2\ldots p_n.$$
\begin{conjecture}
	Let $\varepsilon > 0.$ There are finitely many triples $a, b, c$
	such that
	\begin{enumerate}
		\item $(a,b) = (b,c) = (a, c) = 1,$
		\item $a+b=c,$
	and
		\item $\rad(abc)^{1+\varepsilon}<c.$
	\end{enumerate}
\end{conjecture}
The $abc$ conjectue is thought to lie very deep indeed. For example, Fermat's last theorem for sufficiently large
exponents follows from the $abc$ conjecture \cite{GT}.
That is, for sufficiently large $n,$ the $abc$ conjecture implies
that
	$$a^n + b^n = c^n$$
has no solutions.


A number $x$ is called {\it powerful} if
	$$p | x \Rightarrow p^2 | x.$$
We explore the relationship between powerful numbers and the $abc$ conjecture.

The Brocard problem is to find solutions to the equation 
	$$n! + 1 = m^2.$$
Paul Erd\"os conjectured that this equation has only finitely many solutions. 
In fact only three solutions $(m,n)$ are known:
	$$(4,5), (5,11), (7,71) $$
It was shown, assuming the $abc$ conjecture, \cite{Ov}, that this is indeed the case.
In fact, for a fixed integer $k$,
	$$n! +k = m^2$$
has only finitely many solutions.
We show that, for a fixed $k$, assuming the $abc$ conjecture, $n! +k$ is a powerful number
only finitely often. Qualitatively, powerful numbers cannot be found near 
factorials except finitely often.

Solomon Golomb and Paul Erd{\"o}s seperately conjectured that
there are only finitely many $n$ such that 
	$n,$ $n+1$ and $n+2$ are powerful numbers.
The conjecture has been proven assuming the $abc$ conjecture is true, \cite{Sa}.
We show that for any arithmetic progression of
the form
	$a_n = a + nd$
where $(a,d)=1,$ that there are only finitely many $n$ such that
$a_n$, $a_{n+1}$ and $a_{n+2}$ are powerful.
\section{The radical}
\begin{definition}
Let $x\in\mathbb{N}.$ The radical of $x$, denoted by $\rad(x),$ is the product of the distinct primes which divide $x$.
\\
More concretely, if $x$ has unique prime factorization 
$$x=p_{n_{1}}^{a_{1}}\cdots p_{n_{k}}^{a_{k}}$$
then 
$$\rad(x)=p_{n_{1}}\cdots p_{n_{k}}$$
\end{definition}
We now prove some elementary properties of the radical function. These results mainly make use of the prime factorisation of number.
\begin{lemma}
Let $x,n\in\mathbb{N}.$ Then
$$\rad(x^{n})=\rad(x).$$
\end{lemma} 
\begin{proof}
Let $x$ have unique prime factorisation 
$$x=p_{n_{1}}^{a_{1}}\cdots p_{n_{k}}^{a_{k}}.$$ 
Then
$\rad(x^n)=\rad(p_{n_{1}}^{n\cdot a_{1}}\cdots p_{n_{k}}^{n\cdot a_{k}})=p_{n_{1}}\cdots p_{n_{k}}=\rad(x).$
\end{proof}
The following lemma are proved in a similar manner and their proofs are left as an exercise for the reader.
\begin{lemma}
Let $x,y\in\mathbb{N}.$ Then,
\begin{enumerate}
\item
$\rad(x)\leq x,$
\item
$\rad(xy)\leq\rad(x)\rad(y).$
\end{enumerate}
\end{lemma}


In Section \ref{nearfactorialssection} we make use of the radical of factorials. Before we state what the radical of a factorial number is we need to define what a primorial is.
\begin{definition}
Let $p_{n}$ denote the nth prime number. Let $x\in \mathbb{N}$ and $\pi(x)$ denote the number of primes less than or equal to $x.$ The {\it primorial} of $x$, denoted by $x\#$, is defined as
$$x\# =\Pi_{n=1}^{\pi(x)}p_{n}.$$
In other words the primorial of $x$ is the product of all the primes less than or equal to $x$.
\end{definition} 
\begin{lemma}
Let $n\in\mathbb{N}.$ Then
$$\rad(n!)= n\#.$$
\end{lemma}

We now state a very important lemma concerning the radical of a powerful number.

\begin{lemma}
Let $x$ be a powerful number. Then
$$\rad(x)<x^{\frac{1}{2}}.$$
\end{lemma}
\begin{proof}
Let
$$x=p_{n_{1}}^{a_{1}}\cdots p_{n_{k}}^{a_{k}},$$
where $a_{i}\geq 2$ for each $i$.
\\
\\ 
Then
$$\rad(x)^{2}=p_{n_{1}}^{2}\cdots p_{n_{k}}^{2}.$$
Therefore $\rad(x)^{2}|x$ and so $\rad(x)^2\leq x.$ 
\end{proof}

\section{The abc conjecture revisited}
After exploring the properties of the radical we now return to the abc conjecture. 
\begin{conjecture}
Let $\varepsilon>0.$ Then there exists only finitely many triples $(a,b,c)$ with $(a,b)=1$ and $a+b=c$ such that
$$\rad(abc)^{1+\varepsilon}<c.$$  
\end{conjecture}
We demonstrate the power of the abc conjecture with the following example. 
\begin{example}[Fermat's Last Theorem for sufficiently large exponents]
Let $n\in\mathbb{N}$ be sufficiently large. Then
$$x^{n}+y^{n}=z^{n}$$ 
has no solutions over the integers assuming the abc conjecture. 
\end{example}
\begin{proof}
The proof is left as a guided exercise for the reader to become familiar with the process of applying the abc conjecture.
\\
\\
Suppose that 
$$x^{n}+y^{n}=z^{n}$$ 
for some $n\geq 6$ such that $(x,y)=1.$ Set
$$a=x^n,$$
$$b=y^n,$$
$$c=z^n.$$
\begin{enumerate}
\item
Show that 
$$rad(abc)^2<c$$ using the properties of the radical discussed above.
\\
\\
\item
Apply the abc conjecture with $\varepsilon=1$ and deduce the result. 
\end{enumerate}
\end{proof}

\section{No powerful numbers near factorials}\label{nearfactorialssection}
\begin{theorem}\label{nearfactorials}
	Let $k\geq0.$
	Assuming the $abc$ conjecture,
	there are finitely many $x$ such that $x$ is a powerful number and
		$$|x-n!| \leq k$$
\end{theorem}
We break the proof of Theorem \ref{nearfactorials} into three lemmas below.

\begin{lemma}
	$n!$ is a powerful number finitely often.
\end{lemma}

\begin{proof}
	Let $n>3$. Bertrand's postulate states that there exists a prime number $p$ such that
	$$n<p<2n.$$ 
	(For an elegant and elementary proof of Bertrand's postulate, see \cite{}.)
	Thus, $p$ divides precisely one integer less than $2n.$ This shows that $p|(2n)!$ and $p|(2n-1)!$ and that $p\nmid (2n)!$ and $p\nmid (2n-1)!$. 
	\\
	Thus $n!$ not powerful for each $n\geq 7.$ The rest of the cases can simply be checked by hand. 
\end{proof}

\begin{lemma}
	$n!+k$ is powerful finitely often.
\end{lemma}
\begin{proof}
	We  want to find solutions to
		$$n! + k = x,$$
	where $x$ is a powerful number.	
	
	Let $n \geq k.$
	Note that $k | n!.$
	So finding solutions to the original equation
	is the same as finding integer solutions to
		$$\frac{n!}{k} + 1 = \frac{x}{k}.$$
		
	Let $a = \frac{n!}{k},$
	$b = 1,$
	and $c = \frac{x}{k}.$
	Note that
		$$\rad{\frac{n!}{k}} \leq \rad{n!} \leq n\#,$$
	and
		$$\rad{\frac{x}{k}} \leq \rad{x} \leq x^{1/2}.$$	
	So,
	$$\rad{abc} = 
	\rad{\frac{n!}{k}\frac{x}{k}} \leq
	(n\#) x^{1/2}.$$
	Let $\varepsilon = \frac{1}{2}.$
	$$\rad{abc}^{1+\varepsilon} \leq (n\#)^{3/2} x^{3/4}< x=c$$
	for sufficiently large $n.$	
\end{proof}

We leave the following as an exercise for the reader, which completes the proof of Theorem \ref{nearfactorials}.
\begin{exercise}
	$n!-k$ is powerful finitely often.
\end{exercise}

\section{Powerful numbers occurring in arithmetic progression}
\begin{definition}
Let $(a_{n})_{n\in\mathbb{b}}$ be an arithmetic progression. Suppose there exists a $k$ such that $a_{k},a_{k+1}$ and $a_{k+2}$ are all powerful numbers. Then we say that $(a_{k},a_{k+1},a_{k+2})$ is a {\it powerful triple}. 
\end{definition}
We now study the location of powerful triples inside coprime arithmetic progressions. Suppose that the arithmetic progression in question is the natural numbers. Then Erd{\"o}s, Mollin, and Walsh conjectured that this progressions contains no powerful triples, i.e. there are no three consecutive powerful numbers. The abc-conjecture implies that only finitely many powerful triples can occur in the natural numbers. We now generalise this result to a general coprime arithmetic progression.
\begin{theorem}
Let $(a_{n})$ be a coprime arithmetic progression with common difference $d$. Under the assumption of the abc-conjecture there exists only finitely many powerful triples inside $(a_{n}).$ 
\end{theorem}
\begin{proof}
Let $\rad(d)=N.$ Suppose that $(a_{k},a_{k+1},a_{k+2})$ is a powerful triple such that $a_{k}>N^{5}.$

Note the following inequality

$$a_{k}a_{k+2}=a_{k+1}^{2}-d^2<a_{k+1}^2.$$

 Let $a=d^{2},$ $b=a_{k}a_{k+2}$ and $c=(a_{k+1})^{2}.$ 
\\
Note that
$$a+b = c.$$
We we wish to show that there exists an epislon such that

$$\rad(abc)^{1+\varepsilon}<c,$$
and then invoke the abc-conjecture.
\\
We claim that $\varepsilon=\frac{1}{6}$ works.
\begin{align*}
& \rad(abc)^{\frac{7}{6}}
\\
= & \rad(d^{2}a_{k}a_{k+1}^{2}a_{k+2})^{\frac{7}{6}}
\\
\leq & N^{\frac{7}{6}}\left(\rad(a_{k})\rad(a_{k+1})\rad(a_{k+2})\right)^{\frac{7}{6}}
\\
\leq & N^{\frac{7}{6}}(a_{k}a_{k+1}a_{k+2})^{\frac{7}{12}}
\\
< & N^{\frac{7}{6}} a_{k+1}^{\frac{7}{4}}
\\
< & a_{k+1}^{\frac{7}{6}\cdot\frac{1}{5}}a_{k+1}^{\frac{7}{4}}
\\
= & a_{k+1}^{\frac{119}{60}}
\\
< & a_{k+1}^{2}
\\
= & c.
\end{align*}
\end{proof}

\section{Problems}
We now discuss some exercises the reader is encouraged to try.
\begin{enumerate}
\item
The natural numbers contains infinitely many powerful pairs. Is this true for all coprime arithmetic progressions?
\item
Show that $x+y=z$ where $(x,y)=1$ and $x,y,z$ are all 4-powerful has only finitely many solutions under the assumption of the abc conjecture.
\\
What happens if we take $x,y,z$ to be 3-powerful?
\item
When is $x^{n}+y^{n}$ a powerful number?
\item
Show that $2^{n}+1$ is powerful only finitely often assuming the abc conjecture.
\item
Show that $(n!)^r+k$ is powerful only finitely often. 
\item
We can apply most of our results to a larger set of number, i.e. those who are greater than there radical squared. For example 48 is greater than its radical squared but is not powerful. What other properties does this set have? 
\end{enumerate}


\begin{thebibliography}{150}\label{page:b6}
\bibitem{GT} {\sc A. Granville and T. Tucker}. It's As Easy As abc. {\it Notices of the AMS 49 (10)} (2012), 1224–1231.

\bibitem{Ov} {\sc M. Overholt}. The diophantine equation $n! + 1 = m^2.$ {\it Bull. London Math. Soc., 25 (2)} (1993).

\bibitem{Sa} {\sc M. Saul}. The ABC conjecture.
\end{thebibliography}
\end{document}